\documentclass[oneside,notitlepage,11pt]{article}

\usepackage{amssymb}
\usepackage[leqno]{amsmath}
\usepackage{amsfonts}
\usepackage{amsopn}
\usepackage{amstext}
\usepackage{amsthm}
\usepackage{enumitem}

\usepackage{tikz-cd}
\providecommand{\ar}{\arrow}

\usepackage[colorlinks]{hyperref}

\usepackage{clock} 

\providecommand{\cal}{\mathcal}
\renewcommand{\Bbb}{\mathbb}

\newenvironment{pf}{\begin{proof}}{\end{proof}}

\newcommand{\Tee}{{\cal{T}}}

\newcommand{\Ef}{{\cal{F}}}
\newcommand{\Gee}{{\cal{G}}}

\newcommand{\Kay}{{\cal{K}}}

\newcommand{\Pee}{{\cal{P}}}

\newcommand{\Zee}{{\Bbb{Z}}}

\newcommand{\Qyu}{{\Bbb{Q}}}

\renewcommand{\phi}{\varphi}
\renewcommand{\rho}{\varrho}

\newcommand{\rest}{\restriction}

\newcommand{\ntr}{{n\in\omega}}

\newcommand{\loe}{\leqslant}

\newcommand{\subs}{\subseteq}
\newcommand{\sups}{\supseteq}

\renewcommand{\iff}{\Longleftrightarrow}

\newcommand{\map}[3]{#1\colon #2 \to #3} 

\newcommand{\fra}{Fra\"iss\'e}

\providecommand{\nat}{\omega}

\newcommand{\cmp}{\circ} 

\newcommand{\separator}{\begin{center}***\end{center}}

\newcommand{\proto}[1]{{\mathbb S_\kappa}}

\newtheorem{tw}{Theorem}[section]

\newtheorem{lm}[tw]{Lemma}
\newtheorem{prop}[tw]{Proposition}
\newtheorem{claim}[tw]{Claim}
\theoremstyle{definition}
\newtheorem{df}[tw]{Definition}
\newtheorem{ex}[tw]{Example}

\newtheorem{question}[tw]{Question}

\theoremstyle{remark}

\newcommand{\Age}[1]{\operatorname{Age}(#1)}

\newcommand{\T}{{\Bbb{T}}}

\title{Examples of weak amalgamation classes}
\author{{\sc Adam Krawczyk}\\
	{\small Insititute of Mathematics}\\
	{\small University of Warsaw}\\
	{\small Warsaw, Poland}
\and
{\sc Alex Kruckman}\\
{\small Wesleyan University}\\
{\small Middletown, CT, USA}
\and
{\sc Wies{\l}aw Kubi\'s}\footnote{
	Research of W. Kubi\'s supported by GA\v CR grant EXPRO 20-31529X.}\\
	{\small Institute of Mathematics}\\
	{\small Czech Academy of Sciences}\\
	{\small Prague, Czechia}
\and
{\sc Aristotelis Panagiotopoulos}\\
{\small  Institut f\"ur Mathematische Logik und Grundlagenforschung}\\
{\small  Wilhelms-Universit\"at M\"unster}\\
{\small M\"unster, Germany}
}
\date{\today\ \clocktime}

\begin{document}

\maketitle

\begin{abstract}
We present several examples of hereditary classes of finite structures satisfying the joint embedding property and the weak amalgamation property, but failing the cofinal amalgamation property. These include a continuum-sized family of classes of finite undirected graphs, as well as an example due to Pouzet with countably categorical generic limit. 

\ 


\ 

\end{abstract}


\section{Introduction}\label{sec:intro}

In classical \fra\ theory (see e.g.~\cite[Section 7.1]{Hodges}), the basic objects of interest are \fra\ classes: hereditary classes of finitely generated structures which are countable up to isomorphism and satisfy the joint embedding property and the amalgamation property. Any such class $\mathcal{K}$ is associated to a countable homogeneous structure $\mathcal{M}$, its \emph{\fra\ limit}, which is distinguished up to isomorphism as the unique countable homogeneous structure with age $\mathcal{K}$ (the \emph{age} of a structure $\mathcal{M}$ is the class of isomorphism types of all finitely generated substructures of $\mathcal{M}$).

The \fra\ limit of $\mathcal{K}$ is also generic among countable structures with age contained in $\mathcal{K}$. The precise meaning of generic can be explained in multiple equivalent ways: topologically (the isomorphism class of $\mathcal{M}$ is comeager in an appropriate space of structures)~\cite{PouzetRoux, Cam}, game-theoretically~\cite{KKgames}, or via forcing~\cite{Ziemek}.

It turns out that the full strength of the amalgamation property is not necessary for the existence of a generic structure with age $\mathcal{K}$; a condition called the \emph{weak amalgamation property} suffices and in fact characterizes the existence of a generic structure. The weak amalgamation property was introduced by Ivanov~\cite{Ivanov} (under the name ``almost amalgamation property'') and independently by Kechris and Rosendal~\cite{KecRos}, during the study of generic automorphisms of $\omega$-categorical structures and \fra\ limits.
More recent expositions can be found in~\cite{KKgames} and \cite{KruckPhD}. See also~\cite{KubisWFL} and \cite{DiLiberti} for a purely category-theoretic approach.

Corresponding to this weakening of amalgamation is a weakening of homogeneity: the generic limit, or weak \fra\ limit, of a weak \fra\ class $\mathcal{K}$ is characterized as the unique countable \emph{weakly} homogeneous structure with age $\mathcal{K}$. Weakly homogeneous structures were first studied by Pabion~\cite{Pabion}, who called them prehomogeneous; however, he did not define the weak amalgamation property explicitly. 

Most natural examples of classes with the weak amalgamation property actually satisfy an intermediate condition, the \emph{cofinal amalgamation property}. This property was isolated earlier, by Calais~\cite{Calais}, who also studied the generic limits of classes with the cofinal amalgamation property. Calais called these limits pseudo-homogeneous, but we prefer the term \emph{cofinally homogeneous}. 
The cofinal amalgamation property was also rediscovered by Truss in his work on generic automorphisms~\cite{Truss}.  

In this note, we present several examples of classes of finite structures with the weak amalgamation property but without the cofinal amalgamation property. Each example is a hereditary class in a finite relational language and satisfies the joint embedding property. After making the definitions precise in Section~\ref{sec:defs}, we give an example in Section~\ref{sec:alex} of a class of undirected vertex-colored graphs, which appeared previously in the second-named author's PhD thesis~\cite{KruckPhD}. In Section~\ref{sec:aris}, we give an example of a class of directed edge-colored graphs. The ideas in these examples are refined in Section~\ref{sec:KK} to give an example of a class of undirected graphs, and expanded in Section~\ref{sec:KK2} to a continuum-sized family of examples. Finally, in Section~\ref{sec:pouzet}, we first recall the main aspects of the theory of generic limits of weak \fra\ classes (including weak homogeneity) and next we rescue from obscurity a nice example from~\cite{Pabion}, attributed to Pouzet. This example, unlike our others, has a generic limit with a countably categorical theory. We use this example to answer a question of Ahlman, and we pose a question of our own. 

\tableofcontents

\section{Weak amalgamations}\label{sec:defs}


Throughout this note we consider structures in a countable first-order language. In all of our examples, the language will be finite and relational.

\begin{df}\label{DfWAPs}
	Let $\Ef$ be a class of finitely generated structures.
	\begin{itemize}[itemsep=0pt]
	\item  We say that $\Ef$ has the \emph{hereditary property} (also called \emph{hereditary}), if for every embedding $f\colon Z\to Y$ between structures, if $Y\in \Ef$ then $Z\in \Ef$. Note that every hereditary class is isomorphism-closed. 
	
	\item We say that $\Ef$ has the \emph{joint embedding property} (briefly: \emph{JEP}) if for every $Y,Z\in \Ef$, there exists $X\in \Ef$ and embeddings $f\colon Y\to X$ and $g\colon Z\to X$.
	
	\item We say that $\Ef$ has the \emph{weak amalgamation property} (briefly: \emph{WAP}) if for every $Z \in \Ef$ there is $Z' \in \Ef$ and an embedding $\map \eta {Z}Z'$ such that for all embeddings $\map f {Z'}X$, $\map g {Z'}Y$ with $X,Y \in \Ef$ there exist embeddings $\map{f'}{X}{W}$, $\map{g'}{Y}{W}$ with $W \in \Ef$, satisfying
	$$f' \cmp f \circ \eta = g' \cmp g \circ \eta.$$
	In other words, the following diagram is commutative.
	$$\begin{tikzcd}
	Y \ar[rr, "g'"] & & W \\
	Z' \ar[u, "g"] & & \\
	Z \ar[r, "\eta"] \ar[u, "\eta"] & Z' \ar[r, "f"] & X \ar[uu, "f'"]
	\end{tikzcd}$$
	\item We say that $\Ef$ has the \emph{cofinal amalgamation property} (briefly: \emph{CAP}) if we also require that $$f' \cmp f = g' \cmp g$$ holds in the definition above.
	\item Finally, we say that $\Ef$ has the \emph{amalgamation property} (briefly: \emph{AP}) if we also require that $Z' = Z$ and $\eta = \mathrm{id}_Z$ in the definition above.
	\end{itemize}
\end{df}

A \emph{weak \fra\ class} is a hereditary class of finitely generated structures that has the joint embedding property, the weak amalgamation property and is countable up to isomorphism. 

An isomorphism-closed subclass $\Ef'$ of a class $\Ef$ is called \emph{cofinal} if for every $X \in \Ef$ there exists $Y \in {\Ef'}$ such that $X$ embeds into $Y$. It is easy to see that $\Ef$ has CAP if and only if it has a cofinal subclass $\Ef'$ with AP. It is also clear that AP implies CAP, and CAP implies WAP.
But note that by passing to a cofinal subclass, we lose the hereditary property: the only hereditary cofinal subclass of $\Ef$ is $\Ef$ itself. Among the properties defined above, WAP is the only one which is preserved under passing to and from cofinal subclasses. 

\begin{prop}[cf. {\cite[Prop. 2.7]{KubisWFL}}]
	Let $\Ef$ be a class of finitely generated structures. The following conditions are equivalent.
	\begin{enumerate}[itemsep=0pt]
		\item[{\rm(a)}] $\Ef$ has WAP.
		\item[{\rm(b)}] Every cofinal subclass of $\Ef$ has WAP.
		\item[{\rm(c)}] There exists a cofinal subclass of $\Ef$ that has WAP.
	\end{enumerate}
\end{prop}

\begin{pf}
	(a)$\implies$(b)
	Let $\Gee$ be a cofinal subclass of $\Ef$ and fix $Z \in \Gee$. There is $Z' \in \Ef$ and an embedding $\eta\colon Z\to Z'$, witnessing the WAP in $\Ef$. Since $\Gee$ is cofinal, there is an embedding $\eta'\colon Z'\to Z''$ with $Z''\in \Gee$. Then $\eta'\circ \eta\colon Z\to Z''$ witnesses the WAP in $\Gee$. Indeed, for any embeddings $f\colon Z''\to X$ and $g\colon Z''\to Y$ with $X,Y\in \Gee$, we can apply the WAP in $\Ef$ to the embeddings $f\circ \eta'\colon Z'\to X$ and $g\circ \eta'\colon Z'\to Y$ to obtain $W\in \Ef$ and embeddings $f'\colon X\to W$ and $g'\colon Y\to W$ such that $f'\circ (f\circ \eta')\circ \eta = g'\circ (g\circ \eta')\circ \eta$. Letting $h\colon W\to W'$ be any embedding with $W'\in \Gee$, we have $(h\circ f')\circ f\circ (\eta'\circ \eta) = (h\circ g')\circ g\circ (\eta'\circ \eta)$, as required. 
	
	(b)$\implies$(c) Obvious.
	
	(c)$\implies$(a)
	Let $\Gee$ be cofinal in $\Ef$ and assume $\Gee$ has WAP. Fix $Z \in \Ef$ and choose an embedding $\eta\colon Z\to Z'$ with $Z'\in \Gee$. Now we can find an embedding $\eta'\colon Z'\to Z''$ with $Z''\in \Gee$ witnessing the WAP in $\Gee$. Then $(\eta'\circ \eta)\colon Z\to Z''$ witnesses the WAP in $\Ef$. Indeed, for any embeddings $f\colon Z''\to X$ and $g\colon Z''\to Y$ with $X,Y\in \Ef$, we can find embeddings $f^*\colon X\to X^*$ and $g^*\colon Y\to Y^*$ with $X^*,Y^*\in \Gee$. Then we can apply the WAP in $\Gee$ to the embeddings $f^*\circ f\colon Z''\to X^*$ and $g^*\circ g\colon Z''\to Y^*$ to obtain $W\in \Gee$ and embeddings $f'\colon X^*\to W$ and $g'\colon Y^*\to W$ such that $f'\circ (f^*\circ f)\circ \eta' = g'\circ (g^*\circ g)\circ \eta'$. Then we have $(f'\circ f^*)\circ f\circ (\eta'\circ \eta) = (g'\circ g^*)\circ g\circ (\eta'\circ \eta)$, as required. 
\end{pf}

A simple example shows that CAP can be easily lost when passing to a cofinal subclass.

\begin{ex}
	Let $\Ef$ be the class of all finite acyclic undirected graphs (i.e.\ finite forests). This is an example of a hereditary class satisfying CAP but not AP; in this case, the greatest cofinal subclass with AP is the class of all \emph{connected} finite acyclic undirected graphs (i.e.\ trees).
	Another cofinal subclass is the one consisting of all \emph{disconnected} finite acyclic graphs. This subclass has WAP (because $\Ef$ has WAP) and fails CAP.	
\end{ex}

In fact, given any class of finitely generated structures satisfying CAP and not AP, we may take the cofinal subclass of all structures witnessing the failure of AP. Such a subclass will have WAP and not CAP. On the other hand \emph{hereditary} classes satisfying WAP but not CAP are less easy to come by. The rest of this note is devoted to such examples. All of them are weak \fra\ classes.

\section{The examples}\label{TheXs}


\subsection{A class of undirected labeled graphs}\label{sec:alex}

Let $\Kay$ be the class of all finite acyclic undirected graphs whose vertices are labeled by integers $0,1,2,3,4$ and the following configurations of labelings are omitted:
$$
\begin{tikzcd}
0 \ar[r, dash] \ar[d, dash] & 1 \\
2
\end{tikzcd}
\qquad
\begin{tikzcd}
1 \ar[r, dash] \ar[d, dash] & 2 \\
3
\end{tikzcd}
\qquad
\begin{tikzcd}
2 \ar[r, dash] \ar[d, dash] & 3 \\
4
\end{tikzcd}
\qquad
\begin{tikzcd}
3 \ar[r, dash] \ar[d, dash] & 4 \\
0
\end{tikzcd}
\qquad
\begin{tikzcd}
4 \ar[r, dash] \ar[d, dash] & 0 \\
1
\end{tikzcd}
$$
In other words, a graph $G\in \Kay$ is a finite forest such that each vertex of $G$ has a unique label from the additive group $\Zee/5\Zee$, and if $v \in G$ has label $i \in \Zee/5\Zee$ then it cannot have two neighbors with labels $i+1$ and $i+2$, where $+$ denotes addition in $\Zee/5\Zee$, i.e. addition modulo $5$.
A vertex $v \in G$ is \emph{determined} by a vertex $w \in G$ if $v$ and $w$ are adjacent, $v$ has label $i$, and $w$ has label $i+1$ or $i+2$. Note that in this case $w$ may be determined or not, however it cannot be determined by $v$ (this is why we need at least five labels).

\begin{claim}\label{claim:undetermined}
	Every nonempty $G \in \Kay$ has a vertex that is not determined.
\end{claim}

\begin{pf}
	Supposing every vertex of $G$ is determined, we would be able to construct an infinite path $v_0, v_1, \dots$ in $G$ such that $v_n$ is determined by $v_{n+1}$ for every $n$. The vertices in this path are all distinct, because the relation of being determined is antisymmetric and there are no cycles in $G$. Thus $G$ is infinite, which is a contradiction.
\end{pf}

It is clear that $\Kay$ is hereditary and has JEP.

\begin{claim}
	$\Kay$ does not have CAP.
\end{claim}

\begin{pf}
	Fix a nonempty $G \in \Kay$ and let $v\in G$ be an undetermined vertex, with label $i$. 
	Let $G_1$ and $G_2$ be two extensions of $G$ by adding one new vertex adjacent to $v$ only, such that the new vertex in $G_1$ has label $i+1$, and the new vertex in $G_2$ has label $i+2$.
	Clearly, the inclusions $G \to G_1$ and $G \to G_2$ cannot be amalgamated in $\Kay$.
	
	It follows that for any nonempty $H\in \Kay$, there is no extension $H\subseteq H'\in \Kay$ that serves as a witness for CAP. 
\end{pf}

\begin{claim}\label{claim:wap}
	$\Kay$ has WAP.
\end{claim}

\begin{pf}
	Fix $H \in \Kay$. We find a witness for WAP over $H$ in two stages. First, choose $H\subseteq H'\in \Kay$ such that $H'$ in connected: any two connected components in $H$ can be connected by adding a new vertex, with an appropriate label, adjacent to a single vertex of each of them. Next, choose $H'\subseteq G\in \Kay$ such that every vertex of $H'$ is determined in $G$: for each undetermined vertex $v\in H'$, add a new vertex, with an appropriate label, adjacent only to $v$. Of course, $G\setminus H'$ will contain undetermined vertices. 
	
	Now suppose $\map{e_1}{G}{G_1}$ and $\map{e_2}{G}{G_2}$ are embeddings, with $G_1, G_2 \in \Kay$.
	Let $K$ be the free amalgamation of $e_1\rest H'$ and $e_2 \rest H'$, i.e., with no equalities or edges between vertices in $G_1\setminus e_1(H')$ and $G_2\setminus e_2(H')$. Note that each vertex in $G\setminus H'$ has two distinct images in $K$.
	
	Then $K \in \Kay$. Indeed, $K$ is acyclic because $H'$ is connected and acyclic. Suppose for contradiction that
	$$\begin{tikzcd}
	a \ar[r, dash] & b \ar[r, dash] & c
	\end{tikzcd}$$
	is a subgraph of $K$, where the label of $b$ is $i$, the label of $a$ is $i+1$, and the label of $c$ is $i+2$. This subgraph cannot be contained in $G_1$ or in $G_2$, so we must have $b \in H'$. Then $b$ is determined by a vertex $b'\in G$. Both $G_1$ and $G_2$ contain a copy of $b'$, which gives a contradiction. For example, if the label of $b'$ is $i+1$, then $c$ cannot be adjacent to $b$ in $G_1$ or in $G_2$.
\end{pf}

Finally, note that $\Kay$ is a weak \fra\ class, as it obviously has JEP and countably many isomorphism types.

\subsection{A class of directed graphs with labeled edges}\label{sec:aris}

Our next example is very similar to the previous one. This time we consider directed graphs and label the edges instead of the vertices. 

A {\em directed graph} $G = (A,R)$ is a set $A$ equipped with a binary relation $R\subseteq A\times A$. For an edge $e = (v,w)\in R$, we denote by $s(e) = v$ the \emph{source} of $e$, and by $t(e) = w$ the \emph{target} of $e$. We write $G^*$ for the symmetrization of $G$, defined by $G^* = (A,R^*)$, where $$R^* = R\cup \{(w,v)\mid (v,w)\in R\}.$$ We say that $G = (A,R)$ is a {\em directed forest} if: 
\begin{enumerate}[itemsep=0pt]
	\item[(i)] it is antisymmetric, i.e., if $(v,w)\in R$ then $(w,v)\not\in R$;
	\item[(ii)] it is irreflexive, i.e., $(v,v)\not\in R$;
	\item[(iii)] the undirected graph $G^*$ is acyclic. 
\end{enumerate}

Now consider the language $\mathcal{L}$ consisting of two binary relation symbols $S$ and $T$.
For every $\mathcal{L}$-structure $X=(A,S,T)$, we can form a directed graph $(A,S\cup T)$. Whenever $S$ and $T$ are disjoint, we may think of $(A,S,T)$ as the structure resulting from the directed graph $(A,S\cup T)$ after we color each edge of the latter using two colors.
Consider the class $\mathcal{P}$ of all finite $\mathcal{L}$-structures $X=(A,S,T)$ with the following properties:
\begin{enumerate}[itemsep=0pt]
	\item[(P1)] the directed graph $(A,S\cup T)$ is a directed forest;
	\item[(P2)] the sets $S$ and $T$ are disjoint; 
	\item[(P3)] for every $w\in A$, the set $\{e \in S\cup T\mid t(e) = w\}$ of all edges with target $w$ is entirely contained in either $S$ or $T$.  
\end{enumerate}

It is immediate that $\mathcal{P}$ is hereditary and has JEP. We will now check that $\mathcal{P}$ has WAP but not CAP. 

Let $X = (A,S,T)\in \mathcal{P}$ and let $v\in A$. We say that $v$ is {\em undetermined} if there is no $e\in S\cup T$ with $t(e)=v$. Otherwise $v$ is \emph{determined}. Since $(A,S\cup T)$ is antisymmetric and acyclic, it follows just as in Claim~\ref{claim:undetermined} above that if $A$ is nonempty, then there is some undetermined vertex in $A$. 

\begin{claim}
	 $\mathcal{P}$ does not have CAP. 
\end{claim}

\begin{proof}
	Let $Z\in \mathcal{P}$ be nonempty, and let $v\in Z$ be an undetermined vertex. Then there are two incompatible extensions $Y_1$ and $Y_2$ of $Z$. Define $Y_1$ by adding one vertex $w_1$ to $Z$ and a new $S$-edge from $w_1$ to $v$, and define $Y_2$ by adding one vertex $w_2$ to $Z$ and a new $T$-edge from $w_2$ to $v$. Any amalgamation $X$ of $Y_1$ and $Y_2$ over $Z$ would fail to satisfy property (3) of the definition of $\mathcal{P}$.
	
	It follows that for any nonempty $Z\in \mathcal{P}$, there is no extension $Z\subseteq Z'\in \mathcal{P}$ that would serve as a witness for CAP. 
\end{proof}

\begin{claim}
 $\mathcal{P}$ has WAP.
 \end{claim}

\begin{proof}
The proof is almost exactly like the proof of Claim~\ref{claim:wap} above. 

	Let $Z\in  \mathcal{P}$. 
	We find a witness for WAP over $Z$ in two stages. First, chose an extension $Z\subseteq Z'\in \mathcal{P}$ such that if $Z' = (A,S,T)$, then the undirected graph $(A,S\cup T)^*$ is connected. Next, choose an extension $Z'\subseteq Y\in \mathcal{P}$ such that every vertex of $Z'$ is determined in $Y$. 
	Now suppose $e_1\colon Y\to X_1$ and $e_2\colon Y\to X_2$ are embeddings, with $X_1,X_2\in \mathcal{P}$. Let $W$ be the free amalgamation of $e_1\restriction Z'$ and $e_2\restriction Z'$, i.e., with no equalities or edges between vertices in $X_1\setminus e_1(Z')$ and $X_2\setminus e_2(Z')$. It follows as in the proof of Claim~\ref{claim:wap} that $W\in \mathcal{P}$. 
\end{proof}

Finally, note that $\mathcal{P}$ is a weak \fra\ class, as it obviously has JEP and countably many isomorphism types.

\subsection{A class of undirected graphs}\label{sec:KK}

	Let $\Gee$ be the class of all finite acyclic undirected graphs in which no two vertices of degree greater than $2$ are adjacent.
	Obviously, $\Gee$ is hereditary and has JEP.
	
	We denote by $\deg_G(x)$ the degree of a vertex $x$ in the graph $G$. Being cycle-free ensures that each $G\in \Gee$ has at least one vertex of degree $\loe 1$. 
	Furthermore, every graph in $\Gee$ can be extended to a connected one, simply adding a new vertex connected to selected vertices of degree $\loe 1$ from each component.

\begin{claim}\label{claim:notcap}
	$\Gee$ does not have CAP.
\end{claim}

\begin{pf}
	
	Fix a non-discrete (i.e., containing at least one edge) $H \in \Gee$ and choose $v \in H$ of degree $1$ (such a vertex exists, because $H$ is acyclic).
	Consider two extensions of $H$ by attaching to $v$ one of the following graphs.
	$$
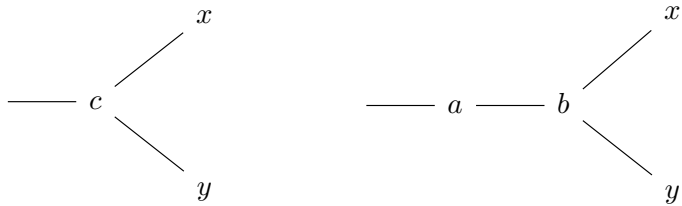
\begin{tikzcd}
	& & x \\
	\ar[r, dash] & c \ar[ru, dash] \ar[rd, dash] & \\
	& & y
	\end{tikzcd}
	\qquad \qquad
	\begin{tikzcd}
	& & & x \\
	\ar[r, dash] & a \ar[r, dash] & b \ar[ru, dash] \ar[rd, dash] & \\
	& & & y
	\end{tikzcd}$$
	Any amalgamation of the resulting graphs over $H$ contains two neighbors of degree $>2$, as either $v$ raises its degree (and $c$ is its neighbor) or else $c$ is identified with $a$ and consequently $a$ and $b$ have degrees $>2$.

	It follows that for any non-discrete $H\in \mathcal{G}$, there is no extension $H\subseteq H'\in \mathcal{G}$ that serves as a witness for CAP. 
\end{pf}
	
\begin{claim}
	$\Gee$ has WAP.
\end{claim}

 \begin{pf}

	Let us call $H \in \Gee$ \emph{tame} if it satisfies the following conditions:
	\begin{enumerate}[itemsep=0pt]
		\item[(1)] $H$ is connected and has more than two vertices.
		\item[(2)] Every vertex of degree 2 in $H$ has a neighbor of degree $>2$.
		\item[(3)] The unique neighbor of every vertex of degree 1 in $H$ has degree 2.
	\end{enumerate}
	We claim that every member of $\Gee$ can be extended to a tame one.
	Fix $H \in \Gee$. It is easy to satisfy (1) by adding new vertices if necessary and extending to a connected graph in $\mathcal{G}$, as noted above. Toward (2), suppose $v\in H$ is a vertex of degree two such that the two neighbors of $v$ have degree $\loe 2$. Then extend $H$ by adding a new vertex adjacent only to $v$. The result is a graph in $\mathcal{G}$ satisfying (1), with fewer vertices of degree $2$ (the new vertex has degree $1$, and $v$ has degree $3$ in the extension). So we can repeat until (2) is satisfied. Toward (3), suppose $v$ is a vertex of degree $1$ whose unique neighbor $w$ does not have degree $2$. By (1), $\deg_H(w) > 2$. Then extend $H$ by adding a new vertex adjacent only to $v$. This preserves (1) and (2), since in the extension $v$ has degree $2$ and a neighbor $w$ of degree $>2$. The new vertex has degree $1$, but its unique neighbor $v$ has degree $2$, so the result is a graph in $\mathcal{G}$ with fewer problematic vertices of degree $1$, and we can repeat until (3) is satisfied.  
	
	We now show WAP. Fix $H \in \Gee$, and let $H'\sups H$ be a tame extension. 
	Let $G \sups H'$ be the graph obtained from $H'$ by adding, for each vertex $u \in H'$ of degree $1$, two new vertices $a_u$, $b_u$ adjacent to $u$. Then $\deg_G(u) = 3$, but by (3), $u$ is not adjacent to any vertex of degree $>2$, so $G\in \mathcal{G}$. 
	
		Now fix two embeddings $\map f G X$, $\map g G Y$ with $X, Y \in \Gee$,
	We claim that there are $W \in \Gee$ and embeddings $\map {f'} X W$, $\map {g'} Y W$ satisfying $f' \cmp f \rest H = g' \cmp g \rest H$. We may assume that $X \cap Y = H'$ and $f\restriction H'$ and $g\restriction H'$ are inclusions.
	Let $W = X \cup Y$ with no new edges between $X \setminus H'$ and $Y \setminus H'$ (so $W$ is the free amalgmation of $X$ and $Y$ over $H'$). Let $f'$, $g'$ be the inclusions $X \subs W$, $Y \subs W$. Clearly, $f' \cmp f \rest H = g' \cmp g \rest H$. It remains to show that $W \in \Gee$.
	
	As $H'$ is connected, $W$ is acyclic.
	Suppose $u,v \in W$ are neighbors of degree $>2$.
	We may assume that $u,v \in X$ (the case $u,v \in Y$ is symmetric).
	Since $X \in \Gee$, we infer that $\deg_X(u) \loe 2$ or $\deg_X(v) \loe 2$.
	Suppose $\deg_X(u) \loe 2$ (the other case is symmetric). Then since $\deg_W(u) > \deg_X(u)$, $u \in H'$, and $\deg_{H'}(u) \loe \deg_X(u) \loe 2$. 
	If $\deg_{H'}(u) = 2$, then by condition (2) of tameness, $u$ is adjacent to a vertex in $H'$ of degree $>2$, so $u$ cannot be adjacent to a vertex of $Y\setminus H'$, contradiction. 
	We conclude that $\deg_{H'}(u) = 1$. But then there are two vertices $a_u,b_u\in G\setminus H'$ which are adjacent to $u$, and $G$ embeds in $X$ over $H'$, so $\deg_{X}(u) \geq 3$, contradiction. 
\end{pf}

As before, note that $\Gee$ is a weak \fra\ class.

\subsection{Continuum-many classes of finite graphs}\label{sec:KK2}

We now present a modification of the example from the previous section, obtaining a different class of finite graphs for each set $A \subs \nat \setminus 3$ with $|A|\geq 2$. Fix such a set $A$ and define $\Gee_A$ to be the class of all finite graphs $G$ satisfying the following conditions.
\begin{enumerate}[itemsep=0pt]
	\item[(1)] The length of each cycle is an element of $A$.
	\item[(2)] No two cycles share an edge. 
	\item[(3)] Each cycle contains at most two vertices of degree $>2$.
\end{enumerate}
Clearly, $\Gee_A$ is hereditary and has JEP.

Given a graph $G \in \Gee_A$, we shall say that a cycle $C \subs G$ is \emph{free} if it contains at most one vertex of degree $>2$ in $G$.
A graph is \emph{non-discrete} if it contains at least one edge.

\begin{lm}\label{KlejmDAFsf}
	Each non-discrete graph in $\Gee_A$ contains either a free cycle or a vertex of degree one.
\end{lm}

\begin{pf}
	Let $G \in \Gee_A$ be non-discrete. Assume $G$ has no free cycles. Then by (3), any cycle in $G$ consists of a pair of paths between two vertices $v$ and $w$ of degree $>2$. By (2), these are the only two paths from $v$ to $w$. Remove both of these paths and the vertices of degree $2$ on them, replacing them by a single edge between $v$ and $w$. This removes a cycle from $G$ and does not change the number of vertices of degree $1$ in $G$, since the degrees of $v$ and $w$ only drop by $1$. After applying this construction to each cycle in $G$, we obtain a non-discrete finite  acyclic graph $G'$, which must have a vertex of degree $1$. Therefore $G$ does as well.
\end{pf}

\begin{lm}\label{lemma:newcycle}
Suppose $G\in \Gee_A$ contains an edge $e$ with endpoints $v$ and $w$ which is not contained in any cycle. Let $n\in A$. Then the graph $H$ obtained by adjoining a new path of length $n-1$ from $v$ to $w$ is in $\Gee_A$. 
\end{lm}
\begin{pf}
Since the edge $e$ is not contained in any cycle in $G$, there is a unique path in $G$ from $v$ to $w$, consisting of the single edge $e$. Thus there is a single new cycle in $H$, namely the cycle of length $n\in A$ consisting of $e$ and the new path. It follows that (1) and (2) are satisfied in $H$. For (3), note that the new cycle contains at most two vertices of degree $>2$, namely $v$ and $w$. And if there is any cycle in $G$ containing $v$ or $w$, this vertex already has degree at least $3$ in $G$ (two from the cycle, together with $e$), so increasing the degree of $v$ and $w$ does not violate (3) for any cycles in $G$. 
\end{pf}

\begin{prop}
	$\Gee_A$ fails the cofinal amalgamation property.
\end{prop}

\begin{pf}
	Fix $G \in \Gee_A$ and assume it is non-discrete. Let $n$ and $m$ be two distinct elements of $A$. If $G$ contains a vertex $v$ of degree $1$, let $w$ be its neighbor, and note that the edge between $v$ and $w$ is not contained in any cycle. Consider the two extensions $G_n$ and $G_m$ of $G$ obtained by adjoining new paths from $v$ to $w$ of lengths $n-1$ and $m-1$, respectively. Then $G_n$ and $G_m$ are both in $\Gee_A$ by Lemma~\ref{lemma:newcycle}. But $G_n$ and $G_m$ cannot be amalgamated over $G$, since in any such amalgam, the two cycles would both contain the original edge between $v$ and $w$, violating (2).
	
	Now suppose $C \subs G$ is a free cycle. $C$ contains at least three vertices, at most one of which has degree $>2$. So let $a$, $b$, and $c$ be distinct elements of $C$ such that $a$ has degree $>2$ if there is such an element in $C$. Let $G_{a,b}$ be the extension of $G$ obtained by adding two new vertices, adjacent to $a$ and $b$, respectively. Let $G_c$ be the extension of $G$ obtained by adding a new vertex adjacent to $c$. Then $G_{a,b}$ and $G_c$ are in $\Gee_A$, but in any amalgamation of these two graphs over $G$, all three of $a$, $b$, and $c$ would have degree $>2$, violating (3).
	
		It follows that for any non-discrete $H\in \mathcal{G}_A$, there is no extension $H\subseteq H'\in \mathcal{G}_A$ that serves as a witness for CAP. 
\end{pf}

\begin{prop}
	$\Gee_A$ has the weak amalgamation property.
\end{prop}

\begin{pf}
	Fix $G \in \Gee_A$. We first extend $G$ to a graph $H \in \Gee_A$ which is connected and non-discrete, and such that every edge in $H$ is contained in a cycle.
	
	By Lemma~\ref{KlejmDAFsf}, each connected component of $G$ contains either a free cycle or a vertex of degree at most $1$. From each component, choose a single vertex which either has degree at most $1$ or has degree $2$ and is contained in a free cycle. Add a single new vertex which is adjacent to each of these chosen vertices. The result is a connected non-discrete graph in $\Gee_A$ containing $G$. Then by repeatedly apply Lemma~\ref{lemma:newcycle}, we can further extend to a graph $H\sups G$ such that every edge in $H$ is contained in a cycle. 
	
	The key property of the graph $H$ is that in any extension $X\sups H$ with $X\in \Gee_A$, any path $p$ in $X$ between distinct vertices $v$ and $w$ in $H$ in already contained in $H$. Indeed, since $H$ is connected, there is some path $q$ from $v$ to $w$ in $H$. If the path $p$ contains an edge which is not in $H$, then there is some cycle $C$ in $X$, consisting of a nonempty segment of $p$ and a nonempty segment of $q$, which is not contained in $H$. But any edge in $q$ is contained in some cycle $C'$ in $H$, so the cycles $C$ and $C'$ share an edge, contradicting (2).  
	
	Now we find a witness for WAP. For each free cycle in $H$, choose one vertex $v$ of degree $2$ in that cycle and add a new vertex adjacent to $v$. If $H$ contains a free cycle with no vertices of degree $>2$ (which only happens if $H$ is itself a cycle), choose two vertices in this cycle and add new vertices adjacent to each of them. Call the resulting graph with no free cycles $H'$. 
	
	Suppose $f\colon H'\to X$ and $g\colon H'\to Y$ are embeddings, with $X,Y\in \Gee_A$. We may assume that $X\cap Y = H$, and that $f\restriction H$ and $g\restriction H$ are inclusions. Let $W = X\cup Y$, with no new edges between vertices in $X\setminus H$ and $Y\setminus H$ (so $W$ is the free amalgamation of $X$ and $Y$ over $H$). It remains to show that $W\in \Gee_A$. 
	
	We observe first that any cycle $C\subseteq W$ which contains more than one vertex in $H$ is contained in $H$. Since there are no edges in $W$ between $X\setminus H$ and $Y\setminus H$, any such cycle $C$ is a union of paths, each contained in $X$ or contained in $Y$, between distinct vertices in $H$. By the key property of $H$ noted above, each of these paths is contained in $H$, so $C\subseteq H$. It follows from this that for any cycle $C\subseteq W$, we have $C\subseteq X$ or $C\subseteq Y$, since $W$ is a free amalgam over $H$. 
	
	From the latter observation, (1) is satisfied in $W$. For (2), note that if two cycles $C$ and $C'$ in $W$ share an edge $e$ between $v$ and $w$, then (without loss of generality) $C\subseteq X$, $C'\subseteq Y$, and $v,w\in H$. But then by the observation above, $C\subseteq H$ and $C'\subseteq H$, contradicting (2) in $H$. 
	
	Finally, for (3), suppose that some cycle $C$ in $W$ contains more than two vertices of degree greater than $2$. We may assume $C\subseteq X$, so there is some vertex $v$ in $C$ such that $\deg_W(v) > 2$ but $\deg_X(v) \leq 2$. Then $v$ is adjacent to a vertex in $Y$, so $v\in H$. Now $H$ is connected and non-discrete, and every edge is part of a cycle, so every vertex in $H$ has degree at least $2$. It follows that $\deg_H(v) = 2$, so $v$ is not adjacent to any vertex in $X\setminus H$. Thus $C$ contains more than one vertex in $H$, so $C\subseteq H$. Since there are no free cycles in $H'$, there are vertices $a$ and $b$ in $C$, distinct from $v$, such that $\deg_{H'}(a)>2$ and $\deg_{H'}(b)>2$. But $H'$ embeds in $Y$ over $H$, so $a$ and $b$ are  two vertices in $C$ with degree $>2$ in $Y$. By (3), we must have $\deg_Y(v) = 2$, contradiction. 
	\end{pf}

\begin{tw}
	There are continuum many different hereditary classes of finite graphs satisfying JEP and WAP but not CAP.
	\end{tw}

\begin{pf}
	For each $A \subs \nat \setminus 3$ with $|A|\geq 2$, the class $\Gee_A$ gives an example, and $\Gee_A \ne \Gee_B$ whenever $A \ne B$.
\end{pf}

\section{Generic limits and Pouzet's example}\label{sec:pouzet}

We conclude by describing an example from~\cite{Pabion}, which Pabion attributes to M.\ Pouzet. This gives an example of a class with WAP but not CAP with a countably categorical generic limit, and it answers a recent question of Ahlman. 

Rather than working with this class directly, it is easier to describe its generic limit. For that reason, we take a moment to define the homogeneity properties alluded to in the introduction. 

\begin{df}
Let $M$ be a countable structure. 
\begin{itemize}[itemsep=0pt]
    \item $\Age{M}$ is the class of finitely generated structures which embed in $M$. 
    \item $M$ is \emph{weakly homogeneous} if for every finitely generated substructure $A\subseteq M$, there exists a finitely generated substructure $A\subseteq B\subseteq M$, such that for any embedding $f\colon B\to M$, $f\restriction A$ extends to an automorphism $\sigma\colon M\to M$. 
    \item $M$ is \emph{cofinally homogeneous} if we also require that $\sigma$ extends $f$ in the definition above. 
    \item $M$ is \emph{homogeneous} if we also require that $A = B$ in the definition above. 
\end{itemize}
\end{df}

\begin{prop}[cf. \cite{KKgames}, \cite{KruckPhD}]\label{prop:homogeneity}
Suppose $\Kay$ is a hereditary class of finitely generated structures with JEP which is countable up to isomorphism (i.e. has countably many isomorphism types).
\begin{itemize}[itemsep=0pt]
    \item $\Kay$ has WAP if and only if there exists a countable structure $M$ which is weakly homogeneous and has $\Age{M} = \Kay$. Moreover, $M$ is unique up to isomorphism. In this case, we call $\Kay$ a \emph{weak \fra\ class} and call $M$ the \emph{generic limit} of $\Kay$. 
    \item $\Kay$ has CAP if and only if its generic limit is cofinally homogeneous.
    \item $\Kay$ has AP if and only if its generic limit is homogeneous. In this case we call $\Kay$ a \emph{\fra\ class} and call $M$ its \emph{\fra\ limit}.  
\end{itemize}
\end{prop}

Generic limits of weak \fra\ classes can also be described in terms of a natural infinite game (called the \emph{Banach-Mazur game} in \cite{KKgames}).
Namely, given a class $\Kay$ and a structure $M$ as above, two players play by alternately choosing bigger and bigger elements of $\Kay$, thus building a chain
$$E_0 \subs E_1 \subs \cdots,$$
 where each inclusion $E_n \subs E_{n+1}$ is an embedding. The second player wins if $\bigcup_{\ntr}E_n$ is isomorphic to $M$.
It turns out that $M$ is the generic limit of $\Kay$ if and only if the second player has a winning strategy in this game, see~\cite{KKgames}.

For instance, if $\Kay$ is the \fra\ class of all finite linearly ordered sets then the winning strategy is very simple and does not even depend on the history: Given a finite linear order $X$, it is enough to add a new point between every two consecutive ones, plus one more point below the minimum and one more above the maximum. By this way, after infinitely many moves the resulting linear order will be dense and with no end points, therefore isomorphic to $(\Qyu, <)$.

Another natural example is the class $\Tee$ of all finite cycle-free graphs. The generic limit is the countable everywhere inifnitely branching tree $\T$, namely, a countable connected cycle-free graph whose each vertex has infinite degree.

The generic limits of our examples are described below.

\begin{ex}
	(a) The generic limit of the class $\Kay$ from Section~\ref{sec:alex} can be obtained from the unique everywhere infinitely branching tree $\T$, by adding labels to all the vertices in such a way that the prohibited configurations do not show up, while at the same time all other configurations appear everywhere. The Banach-Mazur game mentioned above can help here. Indeed, given $E_n \in \Kay$, the second player can always make it connected and can make sure the degree of each vertex of $E_n$ is at least $n$, by adding suitable labels.
	
	(b) The generic limit of the class $\Pee$ from Section~\ref{sec:aris} can be obtained from the tree $\T$ by first turning edges into arrows so that both the source and target degree of each vertex is infinite; next add the ``colors'' ($S$ or $T$) so that condition (P3) is satisfied and both ``colors'' occur infinitely many times in the neighborhood of each vertex.
	
	(c) The generic limit of the class $\Gee$ from Section~\ref{sec:KK} has a particularly nice description. Start with the unique everywhere infinitely branching countable tree $\T$, as above. Pick a subset $X$ of the edges of $\T$ such that for every vertex $v\in \T$, infinitely many edges out of $v$ are in $X$ and infinitely many are not in $X$. Now subdivide each edge $\{v,w\}$ in $X$ into two pieces: $$\begin{tikzcd}
		v \ar[r, dash] & \bullet \ar[r, dash] & w
	\end{tikzcd}$$ and subdivide each edge $\{v,w\}$ not in $X$ into three pieces: $$\begin{tikzcd}
		v \ar[r, dash] & \bullet \ar[r, dash] & \bullet \ar[r, dash] & w
	\end{tikzcd}$$
	The resulting tree is the generic limit of $\mathcal{G}$.
	
	A suitable modification, involving cycles, leads to the description of the generic limit of $\Gee_A$ (see Section~\ref{sec:KK2}). We leave the details to interested readers.
\end{ex}

\separator

We now describe Pouzet's example. Let $R$ be the ternary relation on $\mathbb{Q}$ defined by \[R(x,y,z)\iff x<y \land x < z \land y\neq z.\]
Then the structure $(\mathbb{Q},R)$ is interdefinable with the homogeneous structure $(\mathbb{Q},<)$. It will be useful below to observe that the relation $<$ and its complement are both definable in $(\mathbb{Q},R)$ by existential formulas:
\begin{align*}
x<y & \iff \exists z\, R(x,y,z)\\
\lnot (x<y) & \iff x = y\lor \exists z\, R(y,x,z).
\end{align*}

In~\cite{Pabion}, the follow proposition is stated without proof. 

\begin{prop}\label{prop:pouzet}
The structure $(\mathbb{Q},R)$ is weakly homogeneous but not cofinally homogeneous. 
\end{prop}
\begin{proof}
For weak homogeneity, suppose $A$ is a finite substructure of $\mathbb{Q}$. Let $b$ be any element of $\mathbb{Q}$ which is greater than every element of $A$ in the standard order, and let $B = A\cup \{b\}$. 

Let $g\colon B\to \mathbb{Q}$ be an embedding, and let $f = g\restriction A$. Then $f$ is order-preserving, since for any $a,a'\in A$, we have $a<a'$ if and only if $R(a,a',b)$, if and only if $R(f(a),f(a'),g(b))$, if and only if $f(a) < f(a')$. By homogeneity of the structure $(\mathbb{Q},<)$, $f$ extends to an automorphism of $(\mathbb{Q},<)$, which is also an automorphism of $(\mathbb{Q},R)$.

To contradict cofinal homogeneity, it suffices to show that for any finite substructure $B\subseteq \mathbb{Q}$ with $|B|\geq 2$, there is an embedding $f\colon B\to \mathbb{Q}$ which does not extend to an automorphism of $(\mathbb{Q},R)$. 

Enumerate $B$ in increasing order as $b_1<\dots<b_{n-1}<b_n$, and define $f$ by  $f(b_n) = b_{n-1}$, $f(b_{n-1}) = b_n$, and $f(b_i) = b_i$ for all $1\leq i <n-1$. Then $f$ is an embedding, but it does not extend to an automorphism of $(\mathbb{Q},R)$: all such automorphisms are order-preserving, since $(\mathbb{Q},R)$ is interdefinable with $(\mathbb{Q},<)$.
\end{proof}

It is not hard to show that the age of $(\mathbb{Q},R)$ is the class $K$ of finite structures $(X,R)$ such that for all $x,y,z,w,w'\in X$:
\begin{enumerate}[itemsep=0pt]
    \item If $R(x,y,z)$, then $|\{x,y,z\}| = 3$.
    \item If $R(x,y,z)$, then $R(x,z,y)$.
    \item If $R(x,y,w)$ and $R(y,z,w')$, then $R(x,z,w')$. 
    \item If $|\{x,y,z\}| = 3$, then exactly one of $R(x,y,z)$, $R(y,z,x)$, or $R(z,x,y)$ holds. 
\end{enumerate}
Thus it follows from Proposition~\ref{prop:homogeneity} and Proposition~\ref{prop:pouzet} that $\mathcal{K}$ is a class with WAP but not CAP, whose generic limit is the countably categorical structure $(\mathbb{Q},R)$. 

In~\cite{Ahlman}, Ahlman studied homogenizable structures and introduced the notion of a boundedly homogenizable structure. 

\begin{df}
A structure $M$ in a finite relational language is \emph{homogenizable} if there is a definable expansion $M'$ of $M$ by finitely many new relation symbols, such that $M'$ is homogeneous. $M$ is \emph{boundedly homogenizable} if it is homogenizable and for every finite tuple $a$ from $M$, there exists a finite tuple $b$ from $M$ such that $\mathrm{tp}(ab)$ is isolated by a quantifier-free formula.
\end{df}

Ahlman asked (Question 3.3 in~\cite{Ahlman}) if every model-complete homogenizable structure is boundedly homogenizable. The structure $M = (\mathbb{Q},R)$ provides a negative answer to this question.

The proof of Proposition~\ref{prop:pouzet} shows that $(\mathbb{Q},R)$ is \emph{uniformly} weakly homogeneous. That is, there is a function $f\colon \mathbb{N}\to \mathbb{N}$ such that for every finite substructure $A$ with $|A| = n$, there is a substructure $B$ with $|B| \leq f(n)$ which serves as a witness for weak homogeneity over $A$. In this case, we can take $f(n) = n+1$. 

If $M$ is a countable structure in a finite relational language, then $M$ is uniformly weakly homogeneous if and only if $\text{Th}(M)$ is countably categorical and model complete, see~\cite[Proposition 3]{Pabion}. In the case $M = (\mathbb{Q},R)$, it is also easy to argue directly: $M$ is a reduct of with the structure $(\mathbb{Q},R,<)$, whose theory is countably categorical, so $\text{Th}(M)$ is countably categorical. Further, $\text{Th}(\mathbb{Q},R,<)$ has quantifier elimination, and as observed above, the relation $<$ and its complement are both definable in $M$ by existential formulas. It follows that every formula is equivalent modulo \text{Th}(M) to an existential formula, so $\text{Th}(M)$ is model complete. 

The definable expansion $(\mathbb{Q},R,<)$ also shows that $(\mathbb{Q},R)$ is homogenizable. The argument that $(\mathbb{Q},R)$ is not boundedly homogenizable is essentially the same as the proof that this structure is not cofinally homogeneous. For any tuple $a$ containing at least two distinct elements, $\mathrm{tp}(a)$ is not isolated by a quantifier-free formula, since the partial map exchanging the two greatest elements of $a$ preserves the truth of all quantifier-free formulas but does not preserve the truth of the formula $\exists z\, R(x,y,z)$ expressing the order relation.

We end with a question. All of the examples in this note are closely related to trees or orders: a witness to the weak amalgamation property over $A$ must always go ``further out'' (in the tree or in the order). No infinite tree has a countably categorical theory, which suggests the possibility that for a weak \fra\ class with countably categorical generic limit, a failure of CAP must come from a definable order. Recall that a first-order theory $T$ has the \emph{strict order property} if there is a formula $\varphi(x,y)$, where $x$ and $y$ are tuples of variables of the same length, such that in some model $M\models T$, $\varphi(x,y)$ defines a preorder with infinite chains. 

\begin{question}
Suppose $\mathcal{K}$ is a weak \fra\ class without CAP in a finite relational language, and let $M$ be its generic limit. If $\mathrm{Th}(M)$ is countably categorical, does $\mathrm{Th}(M)$ have the strict order property? Equivalently, if $M$ is a countable structure in a finite relational language which is uniformly weakly homogeneous but not cofinally homogeneous, does $\mathrm{Th}(M)$ have the strict order property?
\end{question}

\end{document}